\documentclass[12pt]{article}
\usepackage[usenames]{color}
\usepackage[latin1]{inputenc}
\usepackage{amsthm,amsfonts,amsmath,amssymb,enumerate,cancel}
 \usepackage{graphicx}
 \usepackage[dvips]{epsfig}

\usepackage{srcltx}

 \usepackage{psfrag}
\newcommand{\esc}[2]{\langle #1,#2\rangle}
\newcommand{\LL}{\Lambda}
\newcommand{\spa}{\mathrm{span}}

\theoremstyle{plain}
\newtheorem{proposition}{Proposition}
\newtheorem{lemma}[proposition]{Lemma}
\newtheorem{teorema}[proposition]{Theorem}
\newtheorem{coro}[proposition]{Corollary}
\newtheorem*{teo}{Theorem}

\theoremstyle{remark}

\newtheorem{remark}[proposition]{Remark}

\newcommand{\red}{}

\theoremstyle{definition}
\newtheorem{definition}{Definition}[section]

\title{The geometry of canal surfaces and the length of curves in
de Sitter space}
\author{R\'{e}mi Langevin, Gil Solanes\footnote{Second author was supported by FEDER/MEC grant number MTM2006-04353 and the Ram\'on y Cajal program.}
}

\begin{document}\thispagestyle{empty}
\maketitle

\abstract{\red{We find the minimal value of the length in de
Sitter space of closed space-like curves with non-vanishing non-space-like geodesic curvature vector. These curves are in correspondence with closed almost-regular canal
surfaces, and their length is a natural magnitude in conformal geometry. 
As an application, we get a lower bound for the total conformal torsion of closed space curves.}

\section{Introduction}
Consider a smooth one parameter family of spheres in ordinary
3-space given by their centers $m(t)$ and radii $r(t)$. These
spheres admit an envelope, classically called a {\em canal surface},
if $\|m'(t)\|> |r'(t)|$. Otherwise the spheres are nested.  
\begin{figure}[ht]
\begin{center}
\psfrag{car}[][]{Characteristic circle} \psfrag{aa}[][]{space-like
path } \psfrag{bb}[][]{time-like path}
{\centering
\mbox{\includegraphics[width=10.8cm,height=5.3cm]{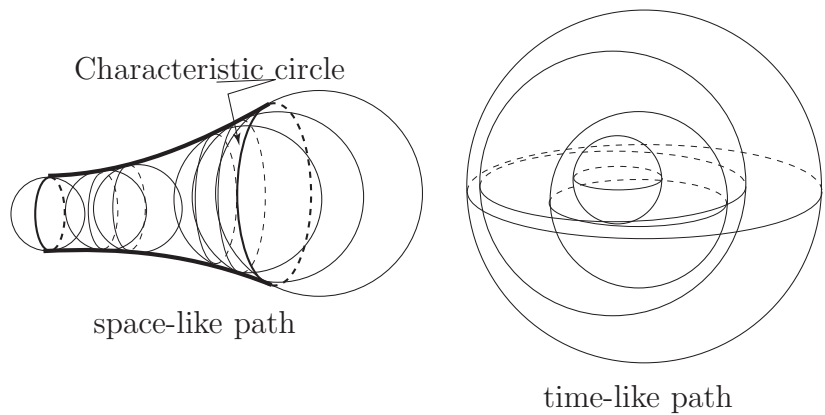}}}
\caption{Spheres corresponding to space-like and  time-like
paths.} \label{type}
\end{center}
\end{figure}
This canal surface is tangent to the
spheres along the so-called \emph{characteristic circles}. These
circles may admit two enveloping curves (possibly coinciding) or
none at all. In case they have no envelope, the canal surface is
immersed, and will be called {\em regular}. Otherwise the surface is
singular along the envelopes of the circles.  {Of special interest is}
the limit case when the two envelopes coincide, and the surface
degenerates along a single curve. 
 In this case, the family of spheres will be called a {\em drill}. 
An example of drill occurs when taking the osculating spheres of a
space curve. 
 In this case, the characteristic circles are the  osculating circles of the curve, so the canal surface is the so-called {\em curvature tube} (cf.\cite{banchoffwhite}). 
 \red{Thomsen (see \cite[\S 8]{Tho}) remarked these are essentially the only examples of drills.}

Back to general families of spheres giving canal surfaces, one can associate them a notion of\emph{ length } in the following way. Given
two nearby spheres of the family (with parameters $t,t'$), consider
the angle $\alpha(t,t')$ between them. Then we take a fine partition
of times $t_1<\ldots<t_n$, and we sum up
$\sum_i\alpha(t_i,t_{i+1})$. The length of the family of spheres is defined to be the supremum of this sum over all partitions.

All these considerations have a conformal invariance in the sense
that they remain unchanged under M\"{o}bius transformations of space.
The space of spheres is naturally endowed with a Lorentz metric
which is invariant under the action of the M\"{o}bius group. This way, it
is isometric to the so-called de Sitter space $\Lambda^4$. It turns out that the
previous discussion is better understood by means of this metric.
Indeed, canal surfaces correspond to curves in the space of
spheres with space-like tangent vector, regular canals are given
by space-like curves with time-like geodesic curvature vector, and  drills correspond to space-like curves with light-like geodesic curvature vector. As for the length described above, it corresponds
precisely to the notion of length associated to the Lorentz metric. Our main result is the following estimate.

\begin{teo}
 Let $\sigma:\mathbb S^1\rightarrow \Lambda^4$ be a closed space-like curve in de Sitter sphere with nowhere vanishing and nowhere space-like geodesic curvature vector. Then the length of $\sigma$ is bigger than or equal to $2\pi$. 
\end{teo}
Moreover, we characterise the equality case. It corresponds to a family of spheres which are tangent to a common line at a fixed point. In particular the length of a closed regular canal is bigger than $2\pi$. The theorem applies also to drills with non-vanishing curvature vector. As an application, by considering the osculating spheres, we obtain an inequality for the integral of the conformal torsion of space curves.


\bigskip
{This work was done while the second author was visiting the Univerist\'e de Bourgogne. He thanks the Institut de Math\'ematiques de Bourgogne for its hospitality. The authors wish to thank Udo Hertrich-Jeromin for pointing to them the reference \cite{Tho}.}

\section{Preliminaries}
We start by recalling a commonly used model for the conformal
geometry of the sphere $\mathbb{S}^3$ (cf. \cite{Ber, pinkall, Ce,
hertrich}). Indeed, for simplicity we will work in $\mathbb{S}^3$
rather that $\mathbb{R}^3$. Let $\langle\ ,\ \rangle$ be the
Lorentz {bilinear}  form defined on $\mathbb{R}^5$ by
$$ \langle{x} ,{y} \rangle =x_1y_1+x_2y_2+x_3y_3+x_4y_4-x_5y_5, $$
and consider the {\it light cone}
$\mathcal{C}=\{x\in\mathbb{R}^5\,|\,\langle x,x\rangle=0 \}$. We
identify $\mathbb{S}^3$ to the projectivization of $\mathcal{C}$.
This way, every vector $\gamma\in \mathcal{C}\setminus \{0\}$
defines a point $\Gamma=\spa(\gamma)\in \mathbb{S}^3$. {Every
section $\gamma:\mathbb{S}^3\rightarrow\mathcal{C}\setminus \{0\}$
transverse to the fibers defines a Riemannian metric on
$\mathbb{S}^3$, conformally equivalent to the standard one. Metrics
of constant sectional curvature correspond to sections of
$\mathcal{C}$ by hyperplanes. Also, conformal transformations of
$\mathbb{S}^3$ are given by linear endomorphisms of $\mathbb{R}^5$
preserving $\langle\ ,\ \rangle$.}

\begin{figure}
\begin{center}
\psfrag{ssx}{$\Sigma$} \psfrag{lxpur}{ $\sigma^{\bot}$}
\psfrag{xx}{ $\sigma$} \psfrag{lx}{ $\spa{(\sigma)}$}
\psfrag{lamb}{$\Lambda^4$} \psfrag{light}{ light cone}
\psfrag{sinf}{ $\mathbb{S}^3$}
\includegraphics[width=6cm]{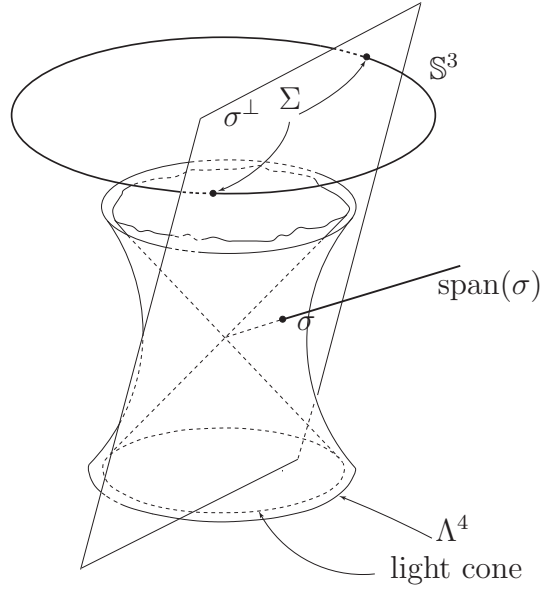}
\end{center}
\caption{$\mathbb{S}^3$ and the correspondence between $\Lambda^4$
and the space of spheres. \label{fig1}}
\end{figure}

The quadric $\Lambda^4=\{x\in\mathbb{R}^5\,|\,\langle x,x\rangle=1
\}$ (sometimes called \emph{de Sitter space}) is identified to the
set of oriented $2$-spheres of $\mathbb{S}^3$. Indeed, every
$\sigma\in\Lambda^4$ defines the sphere
$\Sigma\subset\mathbb{S}^3$ formed by the null lines of
$\sigma^\bot\cap \mathcal{C}$, and an orientation on it ({here and
in the following $\bot$ means} orthogonality with respect to
$\langle\ ,\ \rangle$). The orientation of $\Sigma$ is given by the ball $B_\sigma=\{\spa(\gamma)|\esc{\sigma}{\gamma}>0\}$. For instance, fixing on $\mathbb{S}^3$ the
(standard) metric given by the section $S=\{x\in\mathcal{C}\ |\
x^5=1 \}$, the ball of center $(m_1,\ldots,m_4,1)\in S\equiv
\mathbb{S}^3$ and radius $r\in(0,\pi)$ corresponds to
\begin{equation}\label{centreradi}
\sigma=\frac{1}{\sin r}\left(m_1,\ldots,m_4,\cos
r\right)\in\Lambda^4.
\end{equation}

Given $\sigma_1,\sigma_2\in\Lambda^4$ spanning a space-like plane, the corresponding spheres $\Sigma_1,\Sigma_2$ intersect at the circle $\mathrm{span}(\sigma_1,\sigma_2)^\bot\cap\mathcal C$. A computation using \eqref{centreradi} (cf.\cite[p.39]{hertrich}) shows that the unoriented angle $\alpha\in [0,\pi/2]$ between $\Sigma_1,\Sigma_2$ is given by
\begin{equation}\label{angle}
 \cos\alpha=|\langle\sigma_1,\sigma_2\rangle|.
\end{equation}

\bigskip
We will use the following lemma to find the contact
order between a curve $\Gamma\subset \mathbb{S}^3$,  and a sphere $\Sigma$. {In the following $v^{k)}$
denotes the $k$-th derivative of a parametrized curve $v$.}

\begin{lemma}\label{contacte}
A curve $\Gamma(t)=\spa(\gamma(t))$ has contact of order $\geq k$
with a sphere $\Sigma$ corresponding to $\sigma$ iff
\[
\sigma\bot\spa(\gamma(t),{\gamma}'(t),\ldots,\gamma^{k)}(t))
\]
\end{lemma}
\begin{proof}
The sphere $\Sigma$ is the zero level of the function $f(x)=\langle
x, \sigma\rangle$ defined on some section of $\mathcal C$. Then the
contact of $\Gamma(t)$ and $\Sigma$ has the order of the zero of
$(f\circ\gamma)(t)=\langle \gamma(t),\sigma\rangle$.
\end{proof}

{A consequence is that $\Gamma(t)$ has contact of order $k$ with a
circle $C=\Sigma_1\cap \Sigma_2$ iff
$\sigma_1,\sigma_2\bot\spa(\gamma(t),{\gamma}'(t),\ldots,\gamma^{k)}(t))$.
Indeed, if $\Gamma(t)$ has $k$-th contact with $\Sigma_1$ and
$\Sigma_2$ then it has the same contact with $C$. To see this, consider a (local) coordinate system
$(x_1,x_2,\theta)$ on $\mathbb{S}^3$ such that $x_1$, $x_2$ are the
distances (with respect to some metric) to $\Sigma_1$ and
$\Sigma_2$, respectively. In these coordinates $\Gamma(t)$ is
written $(x_1(t),x_2(t),\theta(t))$. Then we have
$x_i(0)=x_i'(0)=\ldots=x_i^{k)}(0)=0;\ i=1,2.$ Hence $\Gamma(t)$ has
contact of order $k$ with the curve $(0,0,\theta(t))$ which is just
the circle $C$.}

\begin{definition}
A differentiable curve in $\LL^4$ will be called a {\em path}. A
path $\sigma(t)$ in $\Lambda^4$ is called space-like (resp.
time-like or light-like) if its tangent vector $\sigma'(t)$ is such
that $\langle \sigma'(t),\sigma'(t)\rangle>0$ (resp. $<0$, or $=0$).
\end{definition}
A space-like path has a well defined length element given by $\| \sigma'(t)\|dt$ (we denote $\|v\|=\langle v,v\rangle^{1/2}$ for any space-like vector $v$). This length element corresponds to the infinitesimal angle between nearby spheres of the canal, and thus coincides with the description given in the introduction. Indeed, for small $h$
\[
 \|\sigma(t+h)-\sigma(t)\|^2=2(1-\langle \sigma(t+h),\sigma(t)\rangle)=2(1-\cos\alpha)\sim \alpha^2
\]
where $\alpha$ is the angle between the spheres corresponding to $\sigma(t)$ and ${\sigma(t+h)}$.

\medskip The following discussion is well-known.

\begin{proposition}\label{nested}
Let $\sigma(t)$ be a \emph{space-like} path in $\Lambda^4$. Then the
corresponding family of spheres $\Sigma(t)$ admits an envelope,
which is generated by the characteristic circles
$\mathrm{span}(\sigma(t),\sigma'(t))^\bot\cap \mathcal{C}$.  If the
path $\sigma(t)$ is \emph{time-like}, then the spheres $\Sigma(t)$
are nested {(disjoint)}. {If the path is \emph{light-like}, and
$\sigma'(t)\neq 0$, then the spheres are disjoint unless
$\spa(\sigma'(t))$ is constant.}
\end{proposition}

\begin{proof}
{Let $\sigma'(t)$ be space like, then $\langle
\sigma(t),\sigma(t)\rangle=1$ yields $\langle
\sigma(t),\sigma'(t)\rangle=0$. Thus,
$\spa(\sigma(t),\sigma'(t))^\bot$ intersects the light cone
transversely, and defines a circle $C(t)$ in $\mathbb{S}^3$. To
check that the spanned surface is tangent to the spheres, consider a
curve $\Gamma(t)=\spa(\gamma(t))$ with $\Gamma(t)\in C(t)$ for all
$t$; i.e., $\gamma(t)\bot\spa(\sigma(t),\sigma'(t))$. Then \[
\langle\gamma,\sigma\rangle\equiv 0\quad\Rightarrow\quad
\langle\gamma',\sigma\rangle+\cancel{\langle
\gamma,\sigma'\rangle}=0,\] so by Lemma \ref{contacte}, the curve
$\Gamma(t)$ is tangent to $\Sigma(t)$.}
{Let now $\sigma'(t)\neq 0$ be time-like or light-like. Given two
spheres $\Sigma(t_1),\Sigma(t_2)$ we have \[
{\sigma(t_2)-\sigma(t_1)=\int_{t_1}^{t_2}\sigma'(\xi)\mathrm{d}\xi.}\]
{Then for every light-like vector $v\neq 0$} \[
{\langle\sigma(t_2)-\sigma(t_1),v\rangle=\int_{t_1}^{t_2}\langle\sigma'(\xi),v\rangle\mathrm{d}\xi\neq
0}
\]{unless $\sigma'(\xi)\subset\mathrm{span} (v)$.
Except this case,   } $\spa(\sigma(t_1),\sigma(t_2))^\bot
\subset(\sigma(t_2)-\sigma(t_1))^\bot$ does not intersect
$\mathcal{C}\setminus \{0\}$. Hence,
$\Sigma(t_1)\cap\Sigma(t_2)=\emptyset$; i.e., we have a family of
disjoint spheres.}
\end{proof}

\bigskip
Classically, a surface obtained as the envelope of a family of spheres is called a {\em canal surface}. Here we will call  {\em canal} {both a space-like
path and the envelope} of the corresponding one-parameter family of
spheres. This envelope may or not have singular points.

\begin{definition}
A canal $\sigma:(a,b)\rightarrow\Lambda^4$ is called {\em regular} if an immersion of the  cylinder $f:(a,b)\times \mathbb S^1\rightarrow \mathbb S^3$ exists such that $f(t,\cdot)$ maps $\mathbb S^1$ to the characteristic circle $\mathrm{span}(\sigma(t),\sigma'(t))^\bot\cap \mathcal{C}$ for every $t\in(a,b)$.
\end{definition}
These canals have been also called {\em  elliptic} in the literature (cf.\cite{pinkall}). This notion is equivalent to a second order condition on $\sigma(t)$.

\begin{proposition}[cf.\cite{pinkall}]\label{condition}
A canal $\sigma(t)$ in $\Lambda^4$ is regular if and only if the subspace
$\,\spa(\sigma(t),\sigma'(t),\sigma''(t))$ intersects transversely
the light cone.\end{proposition}

Consider $s$ the arc-length parameter on the curve $\sigma(s)$. We
will use dots to indicate derivative with respect to $s$; so for
instance $\langle \dot{\sigma}(s),\dot{\sigma}(s)\rangle=1$. The de Sitter sphere admits an invariant affine connection obtained by Lorentz orthogonal projection of the standard connection of $\mathbb R_1^4$. This way, the
\emph{geodesic curvature vector} $\vec{k}_g(s)$ of $\sigma(s)$ in
$\Lambda^4$ is the orthogonal projection of
$\ddot{\sigma}(s)$ to $T_{\sigma(s)}\Lambda^4=\sigma(s)^\bot$.
Now,
\[
\langle \sigma,\sigma\rangle\equiv 1\Rightarrow
\langle\sigma,\dot{\sigma}\rangle\equiv 0\Rightarrow
\esc{\sigma}{\ddot{\sigma}}+\esc{\dot{\sigma}}{\dot{\sigma}}=0
\]
so that $\esc{\sigma}{\ddot{\sigma}}=-1$. Therefore
$\vec{k}_g=\sigma+\ddot{\sigma}$, and the previous proposition
states that a canal is regular if and only if its geodesic
curvature vector in $\Lambda^4$ is time-like.

%
%

It will be interesting to consider a slightly bigger class of canals.

\begin{definition}\label{almost} A canal $\sigma(s)$ is {\em almost regular} if $\vec{k}_g(s)=\sigma(s)+\ddot{\sigma}(s)$
is non vanishing, and nowhere space-like; i.e.,
\[
\langle \vec{k}_g(s),\vec{k}_g(s)\rangle \leq 0,\qquad
\vec{k}_g(s)\neq 0.
\]A canal with everywhere light-like  $\vec k_g(s)$ is called a {\em drill}.
\end{definition}

\begin{figure}[htb]
\begin{center}
{\centering
\mbox{\includegraphics[width=7cm,height=6cm]{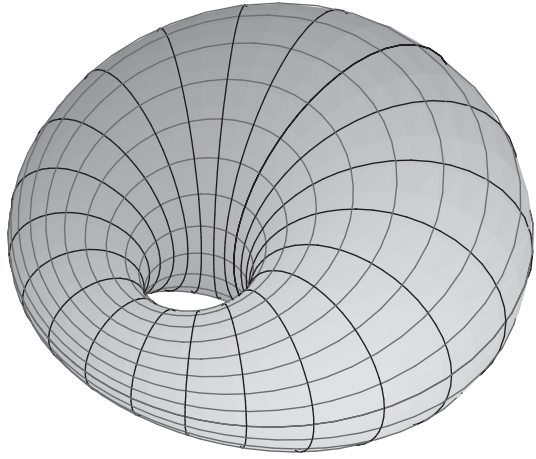}}}
\end{center}
\caption{A Dupin cyclide and its two families of characteristic
circles}
\end{figure}

\section{Examples}\label{examples}
Particular  canal surfaces are the {\em Dupin cyclides}, obtained by intersecting $\Lambda^4$ with space-like affine planes. A reference with our viewpoint
is \cite{La-Wa}. An older one is \cite[livre IV chap XII page
281]{Dar}, where Darboux observes that a Dupin cyclide is in two
different ways the envelope of a one parameter family of spheres,
and that each sphere of one family is tangent to all the spheres
of the other.  Each of these two families is the intersection of $\Lambda^4$ with a space-like affine plane.

\red{ A { regular
Dupin cyclide} (in $\mathbb{S}^3$) is the stereographic image of a
revolution torus of $\mathbb{R}^3$. In this case the two affine planes $H_1$ and $H_2$ are of the form 
$H_1 = x_1 +h_1$, $H_2 = x_2+h_2$, where $h_1$ and $h_2$ are two space-like orthogonal vectorial planes, and $x_1\in H_1$ and $x_2\in H_2$ are two time-like points in $(h_1\oplus h_2)^\bot$ such that $\esc{x_1}{x_2}=1$.
Points $\sigma$ of the intersection $H_i\cap \Lambda^4$ satisfy the equation $\langle \sigma,\sigma \rangle = 1$. Therefore $\langle \sigma-x_i,\sigma-x_i\rangle = \langle \sigma,\sigma \rangle - \langle x_i, x_i \rangle = 1  + |\langle x_i, x_i \rangle|$. This means that the intersection $H_i\cap \Lambda^4$ is, in the euclidean plane $H_i$ (for the metric induced from the Lorentz metric on $\mathbb{R}^4_1$) a circle of radius strictly larger than one. Thus, the two canals have length bigger than $2\pi$.}

\red{A Dupin cyclide may also have two singular points. This happens when one of the family of spheres is of the form $H\cap \Lambda^4$, where $H$ is an affine plane  of the form 
$H = x +h$,  with $h$  a  space-like  vectorial plane, and $x\in H$ a space-like vector  orthogonal to $h$.  Therefore $\langle \sigma-x, \sigma-x \rangle = \langle \sigma,\sigma \rangle - \langle x,x \rangle = 1  - |\langle x,x \rangle|$.
This means that the intersection $H\cap \Lambda^4$ is, in the euclidean plane $H$ (for the metric induced from the Lorentz metric on $\mathbb{R}^4_1$) a circle of radius strictly smaller than one. This radius can be arbitrarily small, so there is no lower bound for the length of such canals.}

\medskip
An example of drill is the  limit case of cyclide shown in picture
(\ref{limcy}). 
\begin{figure}[htb]
\begin{center}
{\centering
\mbox{\includegraphics[width=7cm]{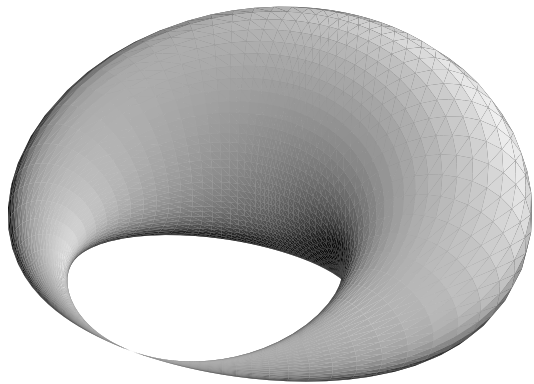}}}
\end{center}
\caption{A degenerate Dupin cyclide \label{limcy}}
\end{figure}
\red{In that case one family of spheres (the exterior
ones) is of the form $H\cap \Lambda^4$, where $H$ is a space-like affine plane  of the form 
$H = x +h$,  where $h$ is a  space-like  vectorial plane, and $x\in H$ a light-like vector orthogonal to $h$.  Therefore $\langle \sigma-x, \sigma-x\rangle= \langle \sigma, \sigma \rangle - \langle x, x\rangle = 1$.
This means that the intersection $H\cap \Lambda^4$ is, in the euclidean plane $H$ (for the metric induced from the Lorentz metric on $\mathbb{R}^4_1$) a circle of radius equal to  one. A degenerate case, for which $\vec{k}_g \equiv 0$  are space-like
geodesics: pencils of spheres containing a given circle.} 

{Both space-like geodesics and Dupin cyclides with one singular point have length $2 \pi$. A more general family with the same
length is the following. Let $u,v,w$ be orthogonal vectors with
$\langle u,u\rangle=0$, $\langle v,v\rangle=\langle w,w\rangle=1$.
Then}\begin{equation}\label{minimes}{\sigma(s)=\lambda(s)u+\cos
sv+\sin s w\quad s\in[0,2\pi]}\end{equation} {for any $2\pi$-periodic
function $\lambda(s)>0$ defines a closed drill with length $2\pi$,
and geodesic curvature vector
$\vec{k}_g(s)=(\lambda(s)+\lambda''(s))u$.
 The corresponding
spheres are tangent to a circle at a fixed point. 

\begin{remark}\label{reciproc}
Reciprocally,
any canal $\sigma(s)$ with light-like geodesic curvature vector of the form
$\vec{k}_g(s)=\rho(s)u$ can be obtained in this way. Indeed, it is
enough to solve the equation $\lambda(s)+\lambda''(s)=\rho(s)$.
Then $\sigma(s)$ has the form (\ref{minimes}) by the uniqueness of
solutions of linear ordinary differential equations.
\end{remark}

\smallskip

The main goal of this paper is to find lower bounds for the length of closed space-like paths.  The example of singular Dupin cyclides shows that some extra geometric hypothesis is needed. We will prove that closed  almost-regular canals have length bigger than $2\pi$. Moreover, it will be seen that equality occurs precisely in the example we just described.

\section{Length of canals}
The conformal geometry of space curves and surfaces was intensively studied at the beginning of 20th century by Blaschke, Thomsen and others, always from a local viewpoint. Next we consider a global
question, namely that of estimating the length in $\Lambda^4$ of
closed canals. 
\begin{figure}
\begin{center}
\psfrag{sig}{$\sigma(s)$}
\psfrag{sigmoins}{$-\sigma(s)$}
\psfrag{phits}{$\phi_t (s)$}
\includegraphics[width=8cm]{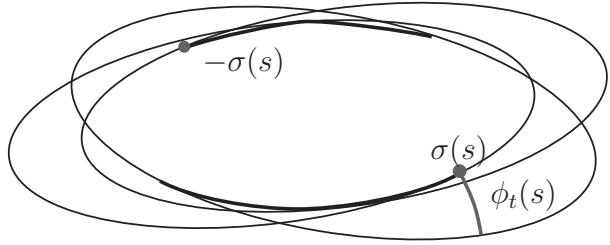}
\caption{Surface spanned by tangent geodesics in $\Lambda^4$ \label{lamba_involute}}
\end{center}
\end{figure}

Let $\sigma(s)$ be an almost regular canal parametrized by arc-length. It will be useful to consider the surface $S\subset\Lambda^4$ spanned by the union of geodesics of the form
$\spa(\sigma(s),\dot{\sigma}(s))\cap\Lambda^4$. This corresponds
to the set of spheres containing the characteristic circles
$C(s)=\spa(\sigma(s),\dot{\sigma}(s))^\bot\cap\mathbb{S}^3$.
Considering the map
\[
\psi(s,u)=\cos u\ \sigma(s)-\sin u\ \dot{\sigma}(s)
\]
we can check that $S$ is a regular surface outside  $\pm \sigma(s)=
\psi(s,n\pi)$, $n\in\mathbb Z$. Now for every $t$ consider the
following path in $\Lambda^4$
\begin{equation}\label{canal}
 \phi_t(s):=\psi(s,s+t)=\cos
(s+t) \sigma(s)-\sin (s+t)\dot{\sigma}(s)
\end{equation}
Taking derivatives,
\begin{equation}\label{derivada}
\phi'_t(s)=\frac{\mathrm{d}\phi_t(s)}{\mathrm{d}s}=-\sin(s+t)(\sigma(s)+\ddot{\sigma}(s))
\end{equation}
so that $\langle \phi'_t,\phi'_t\rangle\leq 0$. The surface $S$ is analogous to the developable surface spanned by the tangent lines of a space curve, while the curves $\phi_t(s)$ mimic the involutes.

Now we prove a couple of technical lemmas. The first one is about
hyperbolic plane geometry.

\begin{lemma}
Let $P,Q,R$ be three points on a geodesic circle in the hyperbolic
plane $\mathbb{H}^2$. Consider the closed convex curve formed by the
geodesic segments $RP$, $RQ$ and an arc of circle joining $P$ and
$Q$. Let $\alpha$, $\beta$, $\gamma$ be the interior angles of
that curve at $P$, $Q$ and $R$ respectively. Then
$\alpha+\beta+\gamma>\pi$.
\end{lemma}
\begin{proof}
Let $\epsilon=\pi-\alpha$ and $\varphi=\pi-\beta$. Then clearly
$\gamma=\pi-\epsilon-\varphi$. Now
\[
\alpha+\beta+\gamma=2\pi-\epsilon-\varphi+\gamma=\pi+2\gamma>\pi.\qedhere
\]
\end{proof}

\begin{figure}
\begin{center}
\psfrag{aa}[][]{$\alpha$}
\psfrag{bb}[][]{$\beta$}
\psfrag{cc}[][]{$\gamma$}
\psfrag{eps}{$\epsilon$}
\psfrag{S1inf}{$S^1_{\infty}$}
\psfrag{phi}{$\varphi$}
\psfrag{PP}{$P$}
\psfrag{QQ}{$Q$}
\psfrag{RR}{$R$}
{\centering \mbox{\includegraphics[width=6cm,height=6cm]{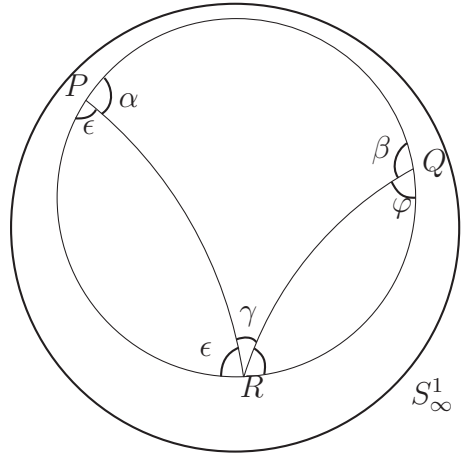}}}
\caption{A lemma of hyperbolic geometry; (Poincar\'{e} model)}
\end{center}
 \end{figure}

Next we bound the length of a regular canal in $\mathbb S^1$. Such an object is a space-like curve $\sigma(t)\in \Lambda^2$, with time-like curvature vector, contained in the $2$-dimensional de Sitter sphere $\Lambda^2$, the unit sphere of $3$-dimensional Lorentz space $\mathbb R_1^2$. As in higher dimension, each $\sigma\in\Lambda^2$ defines an oriented codimension $1$ sphere (i.e. a $0$-sphere in $\mathbb S^1$). By mapping isometrically $\mathbb R_1^2$ into a linear subspace $V\subset\mathbb R_1^4$ one has  $\Lambda^2= V\cap \Lambda^4$. After a suitable stereographic projection, the envelope (in $\mathbb{R}^3$) of a regular canal $\sigma(t)$ in $\Lambda^2$ is a revolution surface. To see that, note that every $\sigma(t)$ corresponds to a sphere orthogonal to the circle $S=V^\bot\cap\mathcal C$. By choosing the pole of the stereographic projection in $S$, the spheres become orthogonal to a line, and are thus enveloped by a surface of revolution. 

\begin{lemma}
Let $\sigma(t)\in \Lambda^2$ with $t\in[0,1]$ be a regular
canal of $\,0$-spheres in $\mathbb{S}^1$ given by their centers
$m(t)\in\mathbb{S}^1$ and their radii $r(t)>0$. Assume {the curve
$m(t)$ has length equal or bigger than $2\pi$}. Assume further
that $r(t)$ is minimum at $t=0,1$ with $r'(0)=r'(1)=0$. Then the
length $\ell$ of $\sigma(t)$ is {strictly} bigger than $2\pi$.
\end{lemma}
\begin{proof}
The canal $\sigma(t)$ defines a one parameter family of geodesics
$\sigma(t)^\bot\cap\mathbb{H}^2$ in hyperbolic plane
$\mathbb{H}^2{=\{x\in\mathbb R_1^2\ |\ \langle x,x\rangle=-1\}}$.
The envelope of these geodesics is given by
\[
c(t)=\frac{1}{\sqrt{-\langle\vec{k}_g,\vec k_g\rangle}}\cdot\vec{k}_g\in \mathbb{H}^2.
\]
Indeed, one easily
checks that $c(t),c'(t)\bot \sigma(t)$. The points $c(t)$
form a regular curve $C\subset\mathbb{H}^2$. Indeed, taking $v(t)=\sigma'(t)/\|\sigma'(t)\|$ one has an orthonormal basis $\sigma(t),v(t),c(t)$. Assuming $c'(t)=0$ for some $t$ easily yields $v'(t)\bot c(t)$. Hence $c\bot\spa(\sigma,v,v')=\mathrm{span}(\sigma,\sigma',\sigma'')$, so the latter subspace is space-like, contradicting Proposition \ref{condition}. 


Since $\sigma(t)$ is a unit normal vector of $C$ at $c(t)$, 
the geodesic curvature of $C$ in $\mathbb{H}^2$ is
$\kappa(t)=\frac{\|\sigma'(t)\|}{\|c'(t)\|}$ and $C$ is locally convex. Thus, the length of
$\sigma$ is the integral of the geodesic curvature of $C$
\[
\ell=\int_C \kappa(t)\|c'(t)\|\mathrm{d} t=\int_C \kappa(t(u))\mathrm{d} u,
\]
where $u$ is the arc-length on $C$.

It will be useful to project $\mathbb H^2$ onto the space-like affine plane $H$ defining the metric on $\mathbb S^1$, and to study the curve $C$ in this instance of the Klein (or projective) model. This curve is given by equations
\[
 x\cdot m(t)=r(t),\qquad x\cdot m'(t)=r'(t),
\]
where $\cdot$ denotes the euclidean scalar product of $H$. Since $r(t)$ is minimum at $t=0,1$ and $r'(0)=r'(1)=0$, the envelope curve $C$ is contained in some disk $D$, and tangent to $\partial D$ at the endpoints. Since $C$ is locally convex, the total turn of its tangent equals the length of $m(t)$, and by assumption is not smaller than $2\pi$. Hence, $C$ must have at least one double tangent, and one double point.  Next we reduce to the case where
$C$ has one single double point (and one doubly tangent geodesic).
Let $t_0$ be the second $t$ such that the geodesic given by
$\sigma(t_0)^\bot$ is tangent to $C$ in more than one point (see
figure \ref{envelope}). Replace an arc of $C$ by the (maximum)
interval of $\sigma(t_0)^\bot$ defined by the contact points. This
way we get a new curve $C'$, and a new canal with a shorter length.
By repeating this procedure with the rest of double tangents, we
arrive at a curve $C'$ with one only double tangent.

\begin{figure}[ht]
\begin{center}
\psfrag{sigto}{$\sigma(t_0 )^{\bot}$}
\psfrag{S1inf}{$S^1_{\infty}$}
\epsfig{file=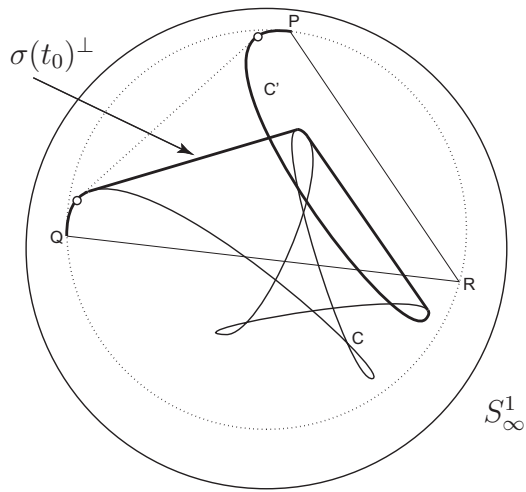,width=0.5\linewidth}%
\caption{Simplifying the curve $C$; (Klein model)\label{envelope}}
\end{center}
\end{figure}

By assumption, $C'$ is contained in a disk of $\mathbb{H}^2$,  and
it is tangent to the boundary at the endpoints $P$, and $Q$. We
can choose $R$ in the boundary of the disk so that the double
point of $C'$ is contained in the angular sector $PRQ$. Adding the
geodesic segments $PR$ and $QR$ to the curve $C'$ we get a locally
convex curve $C''$ with interior angles $\alpha$, $\beta$,
$\gamma$ as in the previous lemma. By applying the Gauss-Bonnet
formula to $C''$ we get
\[
\int_{C'} k(u)\mathrm{d}u +\pi-\alpha+\pi-\beta+\pi-\gamma=4\pi +F
\]
where $F$ is the area of the region enclosed by $C''$, counted
with multiplicity the winding number with respect to $C''$.
Applying the previous lemma gives
\[
\int_{C'} k(u)\mathrm{d}u =\pi+\alpha+\beta+\gamma +F>2\pi+F>2\pi
\]
since the winding numbers are never negative.
\end{proof}


\begin{teorema}\label{main} Let
$\sigma(s)$ be a closed  almost regular canal in $\Lambda^4$. Then
it has length $\ell \geq 2\pi$. {Moreover $\ell=2\pi$ if and only
if $\spa(k_g(s))$ is light-like, and constant.}
\end{teorema}
\begin{proof}
{By Remark \ref{reciproc} we must only prove
that $\ell>2\pi$ if $\spa(\vec{k}_g(s))$ is time-like or
non-constant (say in a neighborhood of $s=0$)}. We will reduce
{this} case to the previous lemma. The procedure is similar to the one used in \cite{pinkall} to prove the Willmore conjecture for
canal surfaces. The first step is to find a conformal image of the
corresponding family of spheres $\Sigma(s)$ so that the curve of the
centers has length $2\pi$ or more.  Consider now the path
$\phi_0(s)$ defined in (\ref{canal}) for all $s\in\mathbb{R}$ (we
think of $\sigma(s)$ as an $\ell$-periodic function defined on
$\mathbb{R}$). By (\ref{derivada}), the tangent vector $\phi_0'(s)$
is a non zero multiple of $\vec{k}_g(s)$ for $s\in(0,\pi)$. Thus, by
Proposition \ref{nested}, $\phi_0(s)$ describes a family of nested
spheres going from $\phi_0(0)=\sigma(0)$ to
$\phi_0(\pi)=-\sigma(\pi)$. Then the oriented spheres $\sigma(0)$
and $\sigma(\pi)$ determine two balls in $\mathbb{S}^3$ that are
disjoint. After a conformal transformation of $\mathbb S^3$, these
two balls will have antipodal centers, and the curve $m(s)$ of the
centers of $\Sigma(s)$ will have length $2\pi$ or more.

Let us now parametrize the curve of centers by its arc-length
parameter $t$ starting at the center of a sphere of minimal
radius. We get a curve $m(t)$ in $\mathbb{S}^3$ with $\|m'(t)\|=1$
and $t\in[0,L]$ for the length $L\geq 2\pi$. We will also denote
by $\Sigma(t)$ and $r(t)$ the spheres and their radii. Consider
now $\widehat{m}(t)$,  a unit speed parametrization of a geodesic
circle of $\mathbb{S}^3$. Consider the canal $\widehat{\Sigma}(t)$
described by the spheres with center $\widehat{m}(t)$,  and radius
$r(t)$ equal to the radius of $\Sigma(t)$. By using
(\ref{centreradi}), the length of the derivative of $\sigma(t)$
with respect to $t$ is given in terms of its centers and radii by
(cf.\cite{pinkall})
\[
\langle
\sigma'(t),\sigma'(t)\rangle=\frac{|m'(t)|^2-r'(t)^2}{\sin^2
r(t)}=\frac{1-r'(t)^2}{\sin^2 r(t)} =
\langle\widehat{\sigma}'(t),\widehat{\sigma}'(t)\rangle
\]
Hence the path $\sigma(t)$ has the same length as the path
$\widehat{\sigma}(t)$ corresponding to the canal
${\widehat{\Sigma}}(t)$.

Let us show that the curve ${\widehat{\sigma}}(t)$ has
(non-space-like) non-zero curvature vector
$\vec{\widehat{k}}_g(t)$. By using again (\ref{centreradi}), as
well as
$$
\sigma''=\frac{\langle\sigma'',\sigma'\rangle}{\langle\sigma',\sigma'\rangle}\sigma'+
\langle\sigma',\sigma'\rangle\ddot{\sigma},
$$
the geodesic curvature vector of $\sigma(t)$ can be computed to be
\[
\vec{k}_g=\sigma+\ddot{\sigma}=\frac{1}{(1-r'^2)^2\sin
r}\left(a,b\, m -ar'\sin r\,m'+(1-r'^2)\sin^2 r\,m''\right )
\]
where $a=(1-r'^2)\cos r-r''\sin r$ and $b=1-r'^2-r''\sin r\cos r$.
For $\widehat{\sigma}$ we have $\widehat{m}''=- \widehat{m}$, and
\[
\vec{\widehat{k}}_g=\frac{a}{(1-r'^2)^2\sin
r}\left(1,\widehat{m}\cos r+\widehat{m}'r'\sin r\right),
\]
so that
\[
\langle
\vec{\widehat{k}}_g,\vec{\widehat{k}}_g\rangle=\frac{-a^2}{(1-r'^2)^3}<0
\]
since $\vec{k}_g\neq 0$ gives $a\neq 0$. Finally we apply the
previous lemma to $\widehat{\sigma}(t)$.
\end{proof}

\begin{remark}
 It is clear from the proof that the same result holds for curves in de Sitter sphere of any dimension.
\end{remark}

\section{Osculating spheres}
Given a curve $\Gamma\subset \mathbb{S}^3$ corresponding to $\gamma
\subset \Lambda^4$, the set of spheres having second order contact
with the curve at $\Gamma(t)$ is the geodesic
$\spa(\gamma,\gamma',\gamma'')^\bot\cap\Lambda^4$ (see Lemma
\ref{contacte}), which corresponds to the pencil of spheres
containing the osculating circle. At least one of these spheres has
a third order contact with $\Gamma$. Indeed,
$\spa(\gamma,\gamma',\gamma'',\gamma''')^\bot$ is not time-like and
must intersect $\Lambda^4$. The spheres having third order contact
with a curve are known as {\em osculating} spheres.

\begin{remark}Some relations between the conformal geometry
of the curve, and  the centers and  radii of the osculating spheres
were found in \cite{sanabria}. \red{Here we focus} in the
interpretation of this canal of oculating spheres as a curve in $\Lambda^4$.
\end{remark}

%

\begin{definition}
A point  is a \emph{vertex} of a curve $\Gamma(t)$ if the the
osculating circle has third order contact with $\Gamma(t)$ at that
point.
\end{definition}
Equivalently, vertices are points where the osculating sphere is not
unique. Performing a stereographic projection
$\psi:\mathbb{S}^3\rightarrow\mathbb{R}^3$, the vertices of
$x(t)=(\psi\circ\Gamma)(t)$ are points where
\[
k'^2(u)+k^2(u)\tau^2(u)=0.
\]
where $k,\tau$ are the curvature and torsion of the curve $x(u)$,
and derivation is taken with respect to the arc-length parameter $u$
of $x(u)$. This can be checked from Bouquet's formula
\begin{multline*}
 x(u+h)-x(u)=\left(h-\frac{k^2(u)h^3}{6}\right )T(u)\\+\left(\frac{k(u)h^2}{2}+\frac{k'(u)h^3}{6}\right)N(u)+\frac{k(u)\tau(u)h^3}{6}B(u)+O(h^4),
\end{multline*}
where $T,N,B$ is the Frenet frame. The same formula yields
the following expression for the center $m(u)$ of the osculating
sphere of $x(u)$ at a non-vertex point
\[
m(u)=x(u)+\frac{1}{k(u)}\vec{n}(u)+\frac{k'(u)}{k^2(u)\tau(u)}\vec{b}(u),
\]
where $\vec{n},\vec{b}$ are the principal normal and binormal
vectors.

By the strong transversality  principle (cf. \cite{arnold}),
vertex-free curves are generical in the sense that they form an
open dense subset in the $C^\infty$ topology.

\begin{remark}
A different definition of the vertices of a space curve can be found in the literature. In
some works vertices are points where the osculating sphere has
contact 4 (cf.\cite{uribe}). There are many examples of vertex-free
curves (in our sense) having contact 4 with its osculating sphere.
For instance, one can consider (non closed) plane curves.
\end{remark}

\begin{definition}
A point where $\Gamma$ has contact 4 (or higher) with a sphere is
called a {\em spherical} point of the curve.
\end{definition}

\bigskip
An unoriented osculating sphere of a curve
$\Gamma(t)=\spa(\gamma(t))$ {(with an arbitrary regular parameter
$t$)} is given by
$\pm\sigma\in\spa(\gamma(t),{\gamma}'(t),{\gamma}''(t),$
${\gamma}'''(t))^{\bot}\cap\Lambda^4$. Clearly vertices of
$\Gamma(t)$ occur when
$\gamma(t),{\gamma}'(t),{\gamma}''(t),{\gamma}'''(t)$ are linearly
dependent. Otherwise the unique  osculating sphere is given (up to
orientation) by
\[
\sigma=\frac{\nu}{\pm\sqrt{\langle \nu,\nu\rangle}}
\]
where
$\nu={\gamma\wedge{\gamma}'\wedge{\gamma}''\wedge{\gamma}'''}$ is
the Lorentz exterior product defined through $\langle \nu,
w\rangle=\det(\gamma,\gamma',\gamma'',\gamma''',w),\ \forall
w\in\mathbb{R}_1^4$ (cf. \cite{langevin.ohara}). Therefore, if
$\Gamma(t)$ is a smooth vertex-free curve, the osculating spheres
give a differentiable path $\sigma(t)$ in $\Lambda^4$ possibly
with $\sigma'(t)=0$ at some point. 

{The following fact was probably known by the time of \cite{Tho}.}
\begin{proposition}\label{osculating}
Let $\Gamma(t)\subset\mathbb{S}^3$ be  a vertex-free curve, and let
$\sigma(t)\in\Lambda^4$ correspond to the osculating spheres $\Sigma(t)$. At 
points where $\sigma'(t)\neq 0$, the geodesic curvature vector
$\vec{k}_g(t)$ is light-like. Moreover, $\vec{k}_g(t)\neq 0$, and
$\Gamma(t)=\spa(\vec{k}_g(t))$. The osculating circles of $\Gamma(t)$ are the characteristic circles of the canal $\Sigma(t)$ given by
$C(t)=\spa(\sigma,\dot{\sigma})^{\bot}\cap\mathcal C$.
\end{proposition}

\begin{proof} Without loss of generality we can consider  $\sigma(s)$ parametrized by arc-length. From
$\sigma\bot\spa(\gamma,\dot{\gamma},\ddot{\gamma},\dddot{\gamma})$ one gets easily $\dot{\sigma}\bot\spa(\gamma,\dot{\gamma},\ddot{\gamma})$,
and $\ddot{\sigma}\bot\spa(\gamma,\dot{\gamma})$. Thus,
$\sigma,\ddot{\sigma}$ belong to
$\spa(\gamma,\dot{\gamma},\dot{\sigma})^\bot$ which has dimension
$2$. Note that $\esc{\sigma}{\ddot\sigma}=-1$. If $\ddot{\sigma}(s)= -\sigma(s)$ at some point, we argue
\[
\esc{\dot{\sigma}}{\ddot{\gamma}}\equiv 0\Rightarrow
\esc{\ddot{\sigma}}{\ddot{\gamma}}+\esc{\dot{\sigma}}{\dddot{\gamma}}=\cancel{\esc{-\sigma}{\ddot{\gamma}}}+\esc{\dot{\sigma}}{\dddot{\gamma}}=0
\]
and
$\dot{\sigma}\bot\spa(\gamma,\dot{\gamma},\ddot{\gamma},\dddot{\gamma})$
giving a second osculating sphere for $\Gamma(t)$. Thus, at
non-vertex points, the vectors $\ddot{\sigma},\sigma$ are linearly independent,
and there must exist  $\beta$ such that
$\sigma+\beta\ddot{\sigma}$ is parallel to $\gamma$. Now
\[
\esc{\sigma+\beta\ddot\sigma}{\sigma}=0\Leftrightarrow \beta=1.
\]Therefore $\vec{k}_g=\sigma+\ddot{\sigma}$ is light-like. Hence $\vec k_g\in \gamma^\bot\cap \mathcal C=\mathrm{span}(\gamma)$, so that $\Gamma(s)=\mathrm{span}(\vec k_g(s))$. Finally, since $\sigma,\dot \sigma\bot\gamma,\dot \gamma,\ddot\gamma$, the circle $C=\mathrm{span}(\sigma,\dot\sigma)^\bot\cap \mathcal C$ has second order contact with $\Gamma$. 
\end{proof}

Roughly speaking we say that the osculating spheres of space curves form a drill. In fact, the converse is essentially true: a drill  with $\vec k_g(t)\neq 0$ corresponds to the oculating spheres of the curve $\Gamma(t)=\mathrm{span}(\vec k_g(t))$. This fact was already observed by Thomsen \cite{Tho}.

The following is an interesting consequence of Proposition \ref{osculating}.
\begin{coro}[] \emph{(cf.\cite{banchoffwhite})}
The osculating circles $C(t)$ of a vertex-free curve $\Gamma(t)$
without spherical points generate a surface $\cup_t C(t)$ (called
{\em curvature tube}) that is immersed outside $\Gamma$.
\end{coro}

\begin{proposition}{\em (cf.\cite[Thm.12.1]{fialkow})}\label{fialkow}
Let $\Gamma(t)=\spa(\gamma(t))$ be a vertex-free curve, and
$\sigma(t)$ its osculating canal. Then $\sigma'(t)=0$ exactly at
spherical points of $\Gamma(t)$ (i.e., where $\Sigma(t)$ has
contact bigger than 3 with $\Gamma(t)$).
\end{proposition}
\begin{proof}
Suppose $\sigma'(t)=0$ for some $t$. Then
\[
\langle \sigma,{\gamma}'''\rangle\equiv 0\Rightarrow
\cancel{\esc{\sigma'(t)}{\gamma'''(t)}}+\esc{{\sigma}(t)}{\gamma^{4)}(t)}=0
\]
so that $\Sigma$ has contact 4 with $\Gamma$. Conversely, if at
some point we have $\sigma\bot\spa(\gamma,\ldots,\gamma^{4)})$
then we get
$\sigma,\sigma'\bot\spa(\gamma,\gamma'\gamma'',\gamma''')$. The
latter space has dimension 4 since $\Gamma$
is vertex-free. Thus $\sigma'$ is parallel to $\sigma$ which yields $\sigma'=0$
since $\sigma'\bot\sigma$.
\end{proof}


The following result is not directly related with our purposes, but may have some interest as a generalisation of the well-known fact that the osculating circles of a planar vertex-free curve are nested (cf.\cite{kneser}). 
\begin{coro}
Given  a sphere $\overline{\Sigma}$ having second contact with a
vertex-free curve $\Gamma(s)$ at $\Gamma(s_0)$, there is a light-like path
$\overline{\Sigma}(s)$ with $\overline{\Sigma}=\overline{\Sigma}(s_0)$, and with contact 2
along $\Gamma(s)$ for $s$ close to $s_0$.
\end{coro}
\begin{proof}Let $\sigma(s)$ be the drill formed by the osculating spheres of $\Gamma(s)$, and consider the canal $\phi_t(s)$ defined in \eqref{canal}. Then we can choose $t$ such that $\phi_t(s_0)$ corresponds to $\overline\Sigma$. Then, equation \eqref{derivada} shows $\spa(\phi_t'(s))=\spa(\sigma(s)+\ddot\sigma(s))=\Gamma(s)$.
Then $\overline{\Sigma}(s)$ is given by $\phi_t(s)$, as one can show a light-like path has always second order contact with the  curve defined by its tangent vector.
\end{proof}

\begin{remark}
The  curves on a given surface such that the osculating sphere is
tangent to the surface were studied by Darboux (cf.\cite{Da}). In particular, he showed they satisfy a (reasonably) nice second order differential equation; at the end of the article he mentions the use of
pentaspherical coordinates as a way to obtain geometrical properties
of these curves. 
\end{remark}

\section{Conformal torsion}
In \cite{sanabria}, a conformally invariant form $\omega$  was
obtained on any vertex-free curve  by pulling back the arc-length
element $\mathrm{d}s$ of the osculating canal $\sigma(s)$. This
form $\omega$ can be expressed in terms of the so-called conformal
torsion and the conformal arc-length. After a stereographic
projection, the computations in \cite{sanabria} give
\[
\omega=\sigma^*(\mathrm{d}
s)=\frac{\sqrt{|m'(u)|^2-r'(u)}}{r(u)}\mathrm{d}u=\frac{|2k'\tau+k^2\tau^3+kk'\tau'-kk''\tau|}{k'^2+k^2\tau^2}\mathrm{d}u
\]
where $m(u), r(u)$ are respectively the center and radius of the
osculating sphere of the projected curve, and derivation is taken
with respect to the arc-length parameter $u$ of this curve. On the
other hand, every vertex-free curve admits a conformally invariant
parameter $t$, called {\em conformal arc-length}, given by
\[
\mathrm{d} t=\sqrt[4]{k'^2+k^2\tau^2}\mathrm{d}u.
\]
Now,
\[
\omega=|T|\mathrm{d}t\qquad \mathrm{where}\qquad
T=\frac{2k'\tau+k^2\tau^3+kk'\tau'-kk''\tau}{(k'^2+k^2\tau^2)^{5/2}}.
\]
The scalar $T$ is a conformally invariant function associated to a vertex-free 
curve and was called {\em conformal torsion} in \cite{csw}. Note that $T$ has a well defined sign, and it vanishes at spherical points by Proposition \ref{fialkow}.
Thus, unlike vertices, spherical points can not be avoided generically. The following result is an application of Theorem \ref{main}.

%
%


\begin{coro}
Let $C\subset\mathbb{R}^3$ be a closed vertex-free curve. If no
sphere has contact $4$ (or higher) with $C$ then
\[
\int_C \omega= \left|\int_C T(t)\mathrm{d}t\right|\geq 2\pi
\]
where $T(\neq 0)$ is the conformal torsion and $t$ denotes the
conformal arc-length.
\end{coro}
\begin{proof}
It remains only to see that the orientations of the osculating
spheres can be taken in a coherent way; so that
$\sigma(\ell)=\sigma(0)$ and not $-\sigma(0)$. Recall that the
osculating spheres form a drill with an envelope $S$ generated by
the osculating circles. This envelope $S$ is an immersed surface
outside the curve $C$. Since $C$, and the osculating circles are oriented,  we can take an orientation on
$S\setminus C$, and push this orientation coherently to the
spheres.
\end{proof}

The integral of $T(t)\mathrm{d}t$ for closed curves was already
considered in \cite{csw} where they prove that it coincides (mod
$2\pi$) with the total torsion $\int \tau\mathrm{d}u$. The total torsion had been
 shown to be invariant mod $2\pi$ under M\"obius transformations by Banchoff and White in
\cite{banchoffwhite}.

It worth mentioning that closed curves without spherical points exist. One example is given by curves in a Dupin cyclide making a suitably chosen constant angle with the characteristic circles (cf.\cite{yun}).

\bigskip
We end with the following question: is the following inequality
true for general closed vertex-free curves with spherical points
\[
\int_C \omega =\int_C |T(t)|\mathrm{d}t\geq 2\pi?
\]

\bibliographystyle{plain}

\end{document}